\documentclass[letterpaper,11pt]{article}

\usepackage{amssymb}
\usepackage{amsmath}
\usepackage{amsthm}
\usepackage{graphicx}
\usepackage{authblk}
\usepackage{xcolor}

\usepackage{float}

\usepackage{booktabs,siunitx}
\usepackage{multirow}

\usepackage{fullpage}
\usepackage{bm}

\usepackage[utf8x]{inputenc}

\usepackage{color}
\usepackage{algpseudocode, algorithm, algorithmicx}

\newtheorem{lem}{Lemma}
\newtheorem{thm}{Theorem}
\newtheorem{cor}{Corollary}

\newtheorem*{remark}{Remark}
\newtheorem*{example}{Example}

\newcommand\amsclass[1]%
  {\hspace{9mm}
   \textbf{MSC codes.  }#1
}

\usepackage{romannum}

\title{An analysis of the derivative-free loss method for solving PDEs}
\author[1]{Jihun Han\thanks{jhan25@albany.edu}}
\author[2]{Yoonsang Lee\thanks{yoonsang.lee@dartmouth.edu}}
\affil[1]{\small Department of Mathematics and Statistics, University at Albany, State University of New York}
\affil[2]{Department of Mathematics, Dartmouth College}
\date{}

\begin{document}
\pagenumbering{arabic}

\maketitle
\begin{abstract}
This study analyzes the derivative-free loss method to solve a certain class of elliptic PDEs and fluid problems using neural networks.
The approach leverages the Feynman–Kac formulation, incorporating stochastic walkers and their averaged values. We investigate how the {time interval} associated with the Feynman–Kac representation and the walker size influence computational efficiency, trainability, and sampling errors.
Our analysis shows that the training loss bias scales proportionally with the {time interval} and the spatial gradient of the neural network, while being inversely proportional to the walker size. Moreover, we demonstrate that the {time interval} must be sufficiently long to enable effective training. These results indicate that the walker size can be chosen as small as possible, provided it satisfies the optimal lower bound determined by the {time interval}. Finally, we present numerical experiments that support our theoretical findings.
\end{abstract}

%

\amsclass{65N15, 65N75, 65C05, 60G46}

\section{Introduction}
Neural networks are widely recognized for their flexibility in approximating complex functions in high-dimensional spaces \cite{cybenko, hornik}. This property has motivated their application to representing solutions of partial differential equations (PDEs). Several neural-network–based approaches have been developed in this context. Physics-Informed Neural Networks (PINNs) \cite{PINN} and the Deep Galerkin Method (DGM) \cite{DGM} employ the strong form of PDEs to define training losses, whereas the Deep Ritz Method (DRM) \cite{DRM} leverages a weak (variational) formulation. Another class of methods is based on stochastic representations of PDEs, including backward stochastic differential equations (BSDEs) and the derivative-free loss method (DFLM) \cite{BSDE, DFLM}. These approaches have demonstrated promising results across a wide range of scientific and engineering applications, particularly in high-dimensional problems where conventional numerical solvers encounter severe limitations \cite{BSDE, DGM, PINNreview}.

\begin{figure}[ht]
\centering
\includegraphics[width=0.7\textwidth]{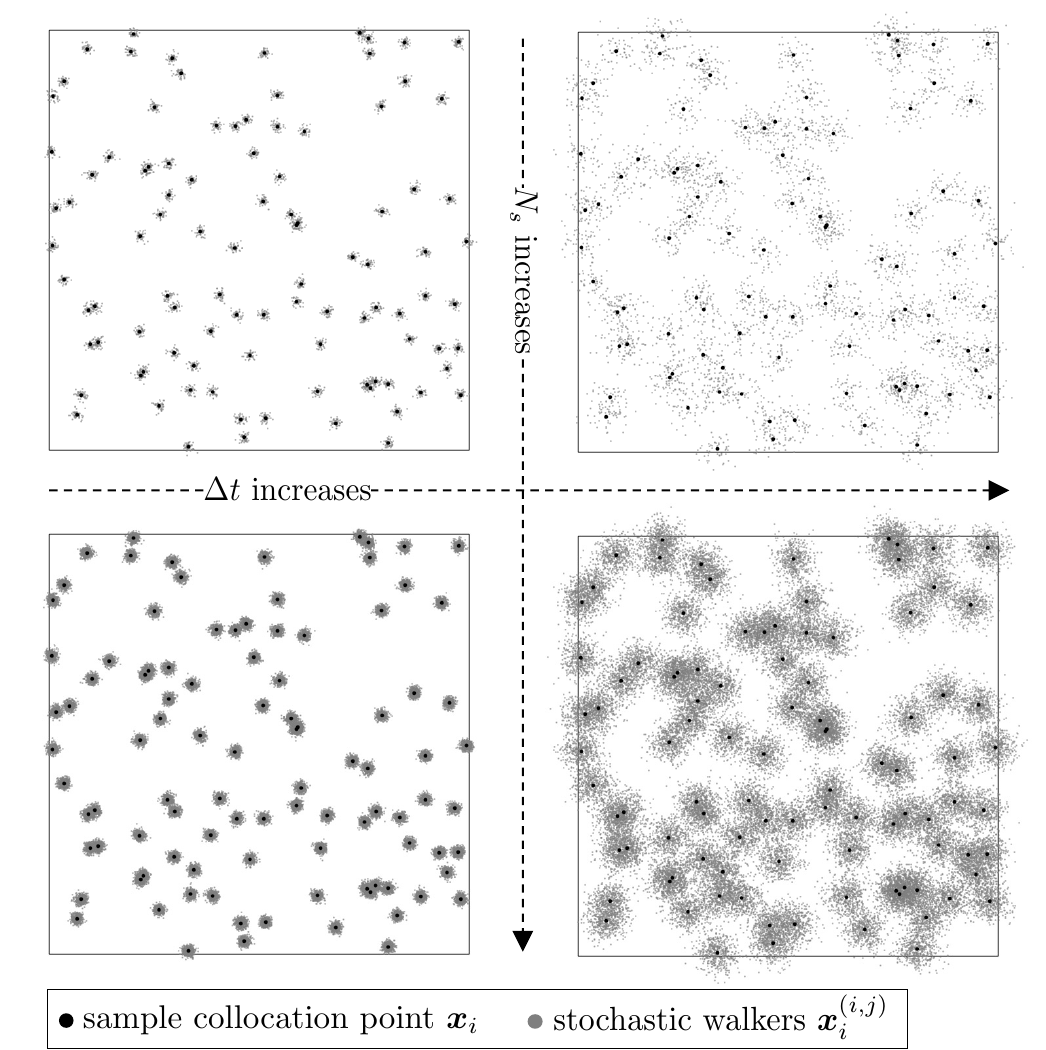}
\caption{Sampling diagram for DFLM}\label{fig:sampling_diagram}
\end{figure}
The present work focuses on the analysis of the Derivative-Free Loss Method (DFLM) \cite{DFLM}. DFLM exploits a stochastic representation of PDE solutions, averaging trajectories of stochastic walkers within a generalized Feynman–Kac framework. Its loss formulation guides a neural network to learn point-to-neighborhood relationships of the solution. Conceptually, DFLM adopts a bootstrapping strategy inspired by reinforcement learning: the target values for training are computed based on the network’s current state through the point-to-neighborhood relation. This iterative scheme incrementally refines the neural network toward solving the PDE. Prior work \cite{DFLM} has shown that this derivative-free formulation offers advantages for handling singularities that arise from geometric features of the domain, such as sharp boundaries. Moreover, the intrinsic averaging mechanism in DFLM has enabled successful applications to homogenization problems \cite{DFLMHomo} and nonlinear flow problems \cite{DFLMflow}.

As with other numerical PDE solvers, DFLM relies on collocation points to enforce constraints. At each collocation point, $N_s$ stochastic walkers are initialized to approximate the expectation in the Feynman–Kac representation. These walkers evolve according to a stochastic process associated with the PDE operator over a short time interval $\Delta t$, which determines the size of a neighborhood to take an average (or expected value). Figure \ref{fig:sampling_diagram} illustrates the role of $N_s$ and $\Delta t$: increasing $N_s$ reduces sampling error by providing more trajectories, while enlarging $\Delta t$ expands the spatial neighborhood explored by the walkers.

A larger $N_s$ decreases variance but increases computational cost, while a longer $\Delta t$ broadens neighborhood coverage but likewise raises the cost of simulation. Our analysis (Theorem \ref{thm:lossbias}) establishes that the empirical training loss has a bias bounded by $\frac{\Delta t}{N_s}$. This result implies that small $N_s$ can be used efficiently, provided $\Delta t$ is kept proportionally small. On the other hand, we prove (Theorem \ref{thm:pointwise_learning}) that $\Delta t$ must exceed a problem-dependent lower bound to ensure that walkers adequately explore neighborhoods; otherwise, the network fails to capture local variations in the solution. Collectively, our analysis highlights a trade-off: $\Delta t$ must be sufficiently large to guarantee learning, while $N_s$ can be chosen as small as possible once this condition is met.

The remainder of the paper is organized as follows. Section \ref{sec:DFLM} reviews the formulation of DFLM and introduces the parameters central to our analysis. Section \ref{sec:analysis} presents the main theoretical results, including bounds on the training loss bias and the conditions under which excessively small $\Delta t$ hampers learning. Section \ref{sec:experiment} provides numerical experiments that validate the theory. Finally, Section \ref{sec:discussion} concludes with a discussion of limitations and potential directions for future research.

\section{Derivative-Free Neural Network Training Method}\label{sec:DFLM}
This section reviews the derivative-free loss method (DFLM) \cite{DFLM}. DFLM addresses PDEs through a stochastic representation inspired by the Feynman–Kac formula, which characterizes how the solution at a point is determined by its surrounding neighborhood. Rather than relying on pointwise PDE residuals, DFLM trains a neural network to directly satisfy these point-to-neighborhood relationships across the domain, thereby recovering the PDE solution. This intrinsic emphasis on interpoint correlation distinguishes DFLM from approaches such as Physics-Informed Neural Networks (PINNs) \cite{PINN}, where the neural network only implicitly learns spatial relationships through residual minimization of the strong-form PDE.

Another distinctive feature of DFLM is its iterative training strategy. The method alternately updates the neural network and the associated target values, in a manner analogous to bootstrapping in reinforcement learning. This contrasts with supervised learning frameworks, where optimization proceeds with fixed targets defined by a prescribed loss function. Through this adaptive update mechanism, DFLM incrementally refines the network toward an accurate PDE solution.

We consider DFLM for the following type of PDEs of an unknown function $u(\bm{x}) \in \mathbb{R}$: 
\begin{equation}\label{eq:Quasi-linear elliptic}
\mathcal{N}[u](\bm{x}):= \frac{1}{2}\Delta u(\bm{x}) + \bm{V}\cdot \nabla_{\bm{x}} u(\bm{x}) - G = 0, ~\textrm{in}~ \Omega \subset \mathbb{R}^k.
\end{equation}
Here $\bm{V}=\bm{V}(\bm{x},u(\bm{x})) \in \mathbb{R}^k$ is the advection velocity and $G=G(\bm{x}, u(\bm{x})) \in \mathbb{R}$ is the force term, both of which can depend on $u$. 

From the standard application of It$\hat{\text{o}}$'s lemma (e.g., in \cite{karatzas}), we have the stochastic representation of the solution of Eq.~\eqref{eq:Quasi-linear elliptic} through the following equivalence;

\begin{itemize}
\item $u:\Omega \rightarrow \mathbb{R}$ is a solution of Eq.~\eqref{eq:Quasi-linear elliptic}. 
\item the stochastic process $q(\Delta t;u, \bm{x}, \{\bm{X}_s\}_{0\leq s\leq {\Delta t}}) \in \mathbb{R}$ defined as 
\begin{equation}\label{eq:q_martingale_stochastic_process}
q({\Delta t};u, \bm{x},\{\bm{X}_s\}_{0\leq s\leq {\Delta t}}) := u(\bm{X}_{\Delta t}) - \int_0^{\Delta t} G(\bm{X}_s,u(\bm{X}_s))ds,\\
\end{equation}
where $\bm{X}_s \in \mathbb{R}^k$ is a stochastic process of the following SDE
\begin{equation}\label{eq:q_stochastic_walkers}
d\bm{X}_s = \bm{V}(\bm{X}_s, u(\bm{X}_s))ds + d\bm{B}_s,\quad \bm{B}_s:\mbox{ standard Brownian motion in }\mathbb{R}^k,
\end{equation}
satisfies the martingale property
\begin{align}\label{eq:q-martingale}
\begin{split}
u(\bm{x})&=q(0;u,\bm{x}, \bm{X}_0)=\mathbb{E}\left[q\left({\Delta t};u,\bm{x}, \{\bm{X}_s\}_{0\leq s\leq {\Delta t}} \right) | \bm{X}_0=\bm{x}\right] \\
&= \mathbb{E}\left[u(\bm{X}_{\Delta t}) - \int_{0}^{{\Delta t}}G(\bm{X}_s, u(\bm{X}_s)) ds \middle |\bm{X}_0=\bm{x}  \right],~\forall\bm{x}\in \Omega, \forall {\Delta t}>0.
\end{split}
\end{align}
\end{itemize}
Regarding to the definition of stochastic process $q(t;u, \bm{x},\{\bm{X}_s\}_{0\leq s\leq {\Delta t}})$, the infinitesimal drift $d(\cdot)$ of the stochastic process $u(\bm{X}_s)$ is connected to the differential operator $\mathcal{N}[u]$ as 
\begin{equation}\label{eq:u_x_ito}
d(u(\bm{X}_s)) = \left(\mathcal{N}[u](\bm{X}_s) + G(\bm{X}_s, u(\bm{X}_s)\right) ds + \nabla u (\bm{X}_s) \cdot d\bm{B}_s.
\end{equation}
The martingale property, Eq.~\eqref{eq:q-martingale}, shows that the solution at a point $\bm{x}$, $u(\bm{X})$ can be represented through its neighborhood statistics observed by the stochastic process $\bm{X}_t$ starting at the point $\bm{x}$ during the time period $[0, {\Delta t}]$. We note that the representation holds for an arbitrary time ${\Delta t}>0$ and any stopping time $\tau$ by the optional stopping theorem \cite{karatzas}. In particular, the exit time from the domain as the stopping time, $\tau=\inf \{s: \bm{X}_s \notin \Omega \}$, induces the well-known Feynman-Kac formula for the PDE \cite{oksendal}. Note that other methods are based on the classical Monte-Carlo of the Feynman-Kac formula \cite{classicFeynman1, classicFeynman2,classicFeynman3, classicFeynman4}. Such methods estimate the solution of a PDE at an individual point independently with the realizations of the stochastic processes $\bm{X}_s$ until it exits from the given domain. DFLM, on the other hand, approximates the PDE solution over the domain at once through a neural network $u(\bm{x};\bm{\theta})$, which is trained to satisfy the martingale property Eq.~\eqref{eq:q-martingale}. In particular, DFLM considers a short period, say $\Delta t$, rather than waiting for whole complete trajectories until $\bm{X}_s$ is out of the domain. This character of DFLM allows a neural network to learn more frequently for a short time period.

DFLM constructs the loss function for training a neural network as
\begin{align}\label{eq:q_loss_continuous}
\mathcal{L}^{\Omega}(\bm{\theta})&=\mathbb{E}_{\bm{x}\sim \Omega}\left[\left|u(\bm{x};\bm{\theta})-\mathbb{E}\left[q\left(\Delta t;u(\cdot;\bm{\theta}),\bm{x}, \{\bm{X}_s\}_{0\leq s\leq \Delta t}\right) | \bm{X}_0=\bm{x}\right] \right|^2\right]\\
&=\mathbb{E}_{\bm{x}\sim \Omega} \left[\left|u(\bm{x};\bm{\theta}) - \mathbb{E}_{\{\bm{X}_s\}_{0\leq s\leq \Delta t}}\left[u(\bm{X}_{\Delta t};\bm{\theta}) - \int_{0}^{\Delta t}G(\bm{X}_s, u(\bm{X}_s;\bm{\theta})) ds \middle |\bm{X}_0=\bm{x}  \right] \right|^2 \right]
\end{align} 
where the outer expection is over the sample collocation point $\bm{x}$ in the domain $\Omega$ and the inner expectation is over the stochastic path $\bm{X}_s$ starting at  $\bm{X}_0=\bm{x}$ during $[0,\Delta t]$. 
In the presence of the drive term $\bm{V}$ that can depend on $u$, the statistics of $\bm{V}$ will be nontrivial and thus a numerical approximation to $\bm{X}_s$ must be calculated by solving Eq.~\eqref{eq:q_stochastic_walkers}. As an alternative to avoid the calculation of the solution to Eq.~\eqref{eq:q_stochastic_walkers}, another martingale process $\tilde{q}({\Delta t};u,\bm{x},\{\bm{B}_s\}_{0\leq s\leq {\Delta t}})$ based on the standard Brownian motion $\bm{B}_s$ is proposed as 
\begin{align}\label{eq:q_tilde_martingale_process}
\tilde{q}({\Delta t};u, \bm{x}&,\{\bm{B}_s\}_{0\leq s\leq {\Delta t}}) := 
 \left(u(\bm{B}_{\Delta t}) - \int_0^{\Delta t} G(\bm{B}_s,u(\bm{B}_s))ds\right) \mathcal{D}(\bm{V},u,{\Delta t}),\\
&\text{where}{\quad}\mathcal{D}(\bm{V},u,{\Delta t})=\exp \biggl(\int^{\Delta t}_0\bm{V}(\bm{B}_s, u(\bm{B}_s))\cdot d\bm{B}_s-\frac{1}{2}\int^{\Delta t}_{0}|\bm{V}(\bm{B}_s, u(\bm{B}_s))|^2 ds \biggr). \nonumber 
\end{align}
Here, $\tilde{q}$-process is of form replacing $\bm{X}_s$ to $\bm{B}_s$ in $q$-process with additional exponential factor $\mathcal{D}(\bm{V}, u,{\Delta t})$ compensating the removal of the drift effect in $\bm{X}_s$ \cite{karatzas, oksendal}. Using the alternative $\tilde{q}$-martingale allows the standard Brownian walkers to explore the domain regardless of the form of the given PDE, which can be drawn from the standard Gaussian distribution without solving SDEs. The alternative loss function corresponding to $\tilde{q}$-martingale is 
\begin{equation}
\mathcal{L}^{\Omega}(\bm{\theta})=\mathbb{E}_{\bm{x}\sim \Omega}\left[\left|u(\bm{x};\bm{\theta})-\mathbb{E}\left[\tilde{q}({\Delta t};u(\cdot;\bm{\theta}),\bm{x}, \{\bm{B}_s\}_{0\leq s\leq \Delta t}) | \bm{B}_0=\bm{x}\right] \right|^2\right].
\end{equation}

In the standard DFLM \cite{DFLM}, for the Dirichlet boundary condition, $u(\bm{x})=g(\bm{x})$ on $\partial \Omega$, we consider $\bm{X}_t$ to be absorbed to the boundary $\partial \Omega$ at the exit position and the value of the neural network is replaced by the given boundary value at the exit position, which makes the information propagate from the boundary into the domain's interior. In this study, to enhance the constraint on the boundary, we add the following boundary loss term 
\begin{equation}
\mathcal{L}^{\partial \Omega}(\bm{\theta}) = \mathbb{E}_{\bm{x}\sim \partial \Omega}\left[|u(\bm{x};\bm{\theta})-g(\bm{x})|^2 \right]
\end{equation}
in the total loss
\begin{equation}
{\mathcal{L}}(\bm{\theta})={\mathcal{L}}^{\Omega}(\bm{\theta}) + {\mathcal{L}}^{\partial\Omega}(\bm{\theta}).
\end{equation}
The loss function ${\mathcal{L}}(\bm{\theta})$ is optimized by a stochastic gradient descent method, and, in particular, the bootstrapping approach is used as the target of the neural network (i.e., the expectation component of $q$- or $\tilde{q}$-process) is pre-evaluated using the current state of neural network parameters $\bm{\theta}$. The $n$-th iteration step for updating the parameters $\bm{\theta}_n$ is 
\begin{equation}\label{eq:sgd_update}
\bm{\theta}_n = \bm{\theta}_{n-1} - \alpha \nabla\widetilde{\mathcal{L}}_n(\bm{\theta}_{n-1}), {\quad}\text{where}{\quad} \widetilde{\mathcal{L}}_n(\bm{\theta})=\widetilde{\mathcal{L}}_n^{\Omega}(\bm{\theta}) + \widetilde{\mathcal{L}}^{\partial\Omega}(\bm{\theta}).
\end{equation}
Here, the term $\widetilde{\mathcal{L}}^{\Omega}(\bm{\theta})$ is the empirical interior loss function using $N_r$ sample collocation points $\{\bm{x}_i\}_{i=1}^{N_r}$ in the interior of the domain $\Omega$, and $N_s$ stochastic walkers $\left\{\bm{X}_s^{(i,j)};s \in [0,\Delta t],  \bm{X}_0=\bm{x}_i\right \}_{j=1}^{N_s}$ at each sample collocation point $\bm{x}_i$, 
\begin{align}\label{eq:q_loss_discrete}
\widetilde{\mathcal{L}}^{\Omega}_{n}(\bm{\theta}) :&= \widetilde{\mathbb{E}}_{\bm{x}\sim \Omega}\left[\left|u(\bm{x};\bm{\theta})-\widetilde{\mathbb{E}}\left[q\left(\Delta t;u(\cdot;\bm{\theta}_{n-1}),\bm{x}, \{\bm{X}_s\}_{0\leq s\leq \Delta t}\right) | \bm{X}_0=\bm{x}\right] \right|^2\right] \\
&= \frac{1}{N_r} \sum \limits_{i=1}^{N_r}\left| u(\bm{x}_i;\bm{\theta})-\frac{1}{N_s}\sum\limits_{j=1}^{N_s}\left\{ u\left(\bm{X}_{\Delta t}^{(i,j)};\bm{\theta}_{n-1} \right)-\int_0^{\Delta t}G\left(\bm{X}_s^{(i,j)}, u\left(\bm{X}_{\Delta t}^{(i,j)};\bm{\theta}_{n-1} \right)\right) ds \right\} \right |^2.
\end{align}
The other term $\widetilde{\mathcal{L}}^{\partial \Omega}(\bm{\theta})$ is the empirical boundary loss function using $N_b$ random boundary collocation points $\{\bm{x}_l\}_{l=1}^{N_b}$ on $\partial \Omega$
\begin{equation}
\widetilde{\mathcal{L}}^{\partial \Omega}(\bm{\theta}) := \widetilde{\mathbb{E}}_{\bm{x}\sim \partial \Omega}\left[|u(\bm{x};\bm{\theta})-g(\bm{x})|^2 \right] = \frac{1}{N_b}\sum \limits_{l=1}^{N_b} \left| u(\bm{x}_l;\bm{\theta}) - g(\bm{x}_l)\right|^2.
\end{equation}

The random interior and boundary collocation points $\{\bm{x}_i\}_{i=1}^{N_r}$ and $\{\bm{x}_l\}_{l=1}^{N_b}$ can follow a distribution whose support covers the domain $\Omega$ and the boundary $\partial \Omega$, respectively. The learning rate $\alpha$ could be tuned at each step and the gradient descent step can be optimized by considering the previous steps, such as Adam optimization \cite{adam}.


\section{Analysis}\label{sec:analysis}
DFLM employs stochastic walkers in the local neighborhoods of collocation points distributed across the domain. Our primary goal in this section is to understand how these walkers influence the training of the neural network. The information incorporated into the network during each iteration is governed by two key parameters: (i) the time interval $\Delta t$ and (ii) the number of stochastic walkers $N_s$. These parameters determine how broadly (via $\Delta t$) and how densely (via $N_s$) the walkers explore the neighborhood of each collocation point (see Fig.~\ref{fig:sampling_diagram}). We demonstrate that the network’s ability to approximate the PDE solution is affected by the bias of the empirical loss function, which depends on both $\Delta t$ and $N_s$. In addition, we show that DFLM requires sufficiently rich neighborhood exploration by the stochastic walkers to capture the local variability of the solution. This leads to the existence of a problem-dependent lower bound $\Delta t^{\ast}$ on the time interval $\Delta t$, below which effective training cannot be guaranteed.

\subsection{Bias in the empirical martingale loss function}

\begin{thm}\label{thm:lossbias} 
The empirical loss $\widetilde{\mathcal{L}}^{\Omega}(\bm{\theta})$ in Eq.~\eqref{eq:q_loss_discrete} is a biased estimator of the exact martingale loss $\mathcal{L}^{\Omega}(\bm{\theta})$ in Eq.~\eqref{eq:q_loss_continuous}. Moreover, when $u(\cdot;\bm{\theta})$ has a small PDE residual $\mathcal{N}[u(\cdot;\bm{\theta})]$ in Eq.~\eqref{eq:Quasi-linear elliptic}, the bias is proprotional to $\Delta t$ and the $\mathcal{L}^2$-norm of $\nabla_{\bm{x}}u(\cdot;\bm{\theta})$ with respect to the sampling measure $\bm{x}\sim \Omega$, while the bias is inversely proportional to $N_s$
\begin{equation}
\normalfont{\mbox{Bias}}_{\mathcal{L}^{\Omega}}\left[\widetilde{\mathcal{L}}^{\Omega} \right]{\quad}\propto{\quad} \frac{\Delta t}{N_s}\mathbb{E}_{\bm{x}\sim \Omega}\left[|\nabla_{\bm{x}} u(\bm{x};\bm{\theta})|^2\right].
\end{equation}
\end{thm}

\begin{proof}\renewcommand{\qedsymbol}{}
For notational simplicity, for a fixed $\Delta t$, we use $y_{\bm{x}}$ for the target random variable 
\begin{equation}
y_{\bm{x}}:=q\left(\Delta t;u(\cdot;\bm{\theta}),\bm{x}, \{\bm{X}_s\}_{0\leq s\leq \Delta t}\right)
\end{equation}
where the subscript $\bm{x}$ describes for the intial value of the stochastic process $\bm{X}_0=\bm{x}$.
We also denote the unbiased sample mean statistic for the target as $\overline{y_{\bm{x}}}$
\begin{equation}
\overline{y_{\bm{x}}}:=\mathbb{E}\left[q\left(\Delta t;u(\cdot;\bm{\theta}_{n-1}),\bm{x}, \{\bm{X}_s\}_{0\leq s\leq \Delta t}\right) | \bm{X}_0=\bm{x}\right].
\end{equation}
We denote the sampling measure of $\bm{x}$ as $\mathbb{P}(\bm{x})$ and the distribution of $\overline{y_{\bm{x}}}$ conditioned on $\bm{x}$ as $\mathbb{P}_{\bm{\theta}}(\overline{y_{\bm{x}}}|\bm{x})$ where the subcript $\bm{\theta}$ corresponds to the dependency of the distribution on the neural network's state.
We now take the expectation of the empirical loss with respect to $\bm{x}$, which yields
\begin{align}
\mathbb{E}\left[\widetilde{\mathcal{L}}^{\Omega}(\bm{\theta})\right] &= \int_{\bm{x}}\left(\int_{\overline{y_{\bm{x}}}}(u(\bm{x};\bm{\theta})-\overline{y_{\bm{x}}})^2 \mathbb{P}_{\bm{\theta}}(\overline{y_{\bm{x}}}|\bm{x}) d\overline{y_{\bm{x}}} \right)\mathbb{P}(\bm{x})d\bm{x} \\
&=\int_{\bm{x}}u^2(\bm{x};\bm{\theta})\mathbb{P}(\bm{x})d\bm{x}-\int_{\bm{x}}2u(\bm{x};\bm{\theta})\mathbb{E}_{\mathbb{P}_{\bm{\theta}}(\cdot|\bm{x})}[\overline{y_{\bm{x}}}]\mathbb{P}(\bm{x})d\bm{x} + \int_{\bm{x}}\mathbb{E}_{\mathbb{P}_{\bm{\theta}}(\cdot|\bm{x})}\left[\overline{y_{\bm{x}}}^2 \right]\mathbb{P}(\bm{x})d\bm{x} \\
&= \int_{\bm{x}}\left(u(\bm{x};\bm{\theta})-\mathbb{E}_{\mathbb{P}_{\bm{\theta}}(\cdot|\bm{x})}[\overline{y_{\bm{x}}}]\right)^2\mathbb{P}(\bm{x})d\bm{x} + \int_{\bm{x}}\left(\mathbb{E}_{\mathbb{P}_{\bm{\theta}}(\cdot|\bm{x})}\left[\overline{y_{\bm{x}}}^2\right]-\mathbb{E}_{\mathbb{P}_{\bm{\theta}}(\cdot|\bm{x})}[\overline{y_{\bm{x}}}]^2 \right) \mathbb{P}(\bm{x})d\bm{x} \\
&= \mathbb{E}_{\bm{x}\sim\mathbb{P}}\left[\left(u(\bm{x};\bm{\theta})-\mathbb{E}_{\mathbb{P}_{\bm{\theta}}(\cdot|\bm{x})}[\overline{y_{\bm{x}}}]\right)^2 \right] + \mathbb{E}_{\bm{x}\sim\mathbb{P}}\left[ \mathbb{V}_{\mathbb{P}_{\bm{\theta}}(\cdot|\bm{x})}(\overline{y_{\bm{x}}})\right]\\
&= \mathcal{L}^{\Omega}(\bm{\theta}) + \mathbb{E}_{\bm{x}\sim\mathbb{P}}\left[ \mathbb{V}_{\mathbb{P}_{\bm{\theta}}(\cdot|\bm{x})}(\overline{y_{\bm{x}}})\right],
\end{align}
which implies that the empirical loss $\widetilde{\mathcal{L}}^{\Omega}(\bm{\theta})$ estimates the exact martingale loss $\mathcal{L}^{\Omega}(\bm{\theta})$ with the bias $\mathbb{E}_{\bm{x}\sim\mathbb{P}}\left[ \mathbb{V}_{\mathbb{P}_{\bm{\theta}}(\cdot|\bm{x})}(\overline{y_{\bm{x}}})\right]$. 

When the PDE residual of $u(\cdot;\bm{\theta})$, $\mathcal{N}[u(\cdot;\bm{\theta})]$, is sufficiently small, the stochastic process $q(\Delta t;$ $u(\cdot;\bm{\theta}),\bm{x}, \{\bm{X}_s\}_{0\leq s\leq \Delta t})$ corresponding to $y_{\bm{x}}$ can be approximated as follows
\begin{align}
d(q\left(\Delta t;u(\cdot;\bm{\theta}),\bm{x}, \{\bm{X}_s\}_{0\leq s\leq \Delta t}\right)) &= (\mathcal{N}[u](\bm{X}_s)) ds + \nabla_{\bm{x}} u (\bm{X}_s;\bm{\theta}) \cdot d\bm{B}_s \\
&\simeq  \nabla_{\bm{x}} u (\bm{X}_s;\bm{\theta}) \cdot d\bm{B}_s.
\end{align}
The variance of $y_{\bm{x}}$ is 
\begin{align}
\mathbb{V}[y_{\bm{x}}] &= \mathbb{V}\left[\int_0^{\Delta t} \nabla_{\bm{x}} u(\bm{X}_s;\bm{\theta})\cdot d\bm{B}_s \middle | \bm{X}_0=\bm{x} \right] \\
&= \mathbb{E}\left[\left(\int_0^{\Delta t} \nabla_{\bm{x}} u(\bm{X}_s;\bm{\theta})\cdot d\bm{B}_s \right)^2 \middle | \bm{X}_0=\bm{x} \right] \\
&= \mathbb{E}\left[\int_0^{\Delta t} |\nabla_{\bm{x}} u(\bm{X}_s;\bm{\theta})|^2 ds  \middle | \bm{X}_0=\bm{x} \right] {\quad}{\quad}~(\because \text{It$\hat{\text{o}}$ isometry}) \\
& \simeq \left |\nabla_{\bm{x}} u (\bm{x};\bm{\theta}) \right|^2\Delta t \hspace{5mm}( \Delta t \ll 1)
\end{align}
Therefore the variance of the sample mean $\overline{y_{\bm{x}}}$ of $y_{\bm{x}}$ is approximated as
\begin{equation}\label{eq:each_target_sample_mean_variance}
\mathbb{V}[\overline{y_{\bm{x}}}] = \frac{1}{N_s}\mathbb{V}[y_{\bm{x}}] \simeq \frac{\left |\nabla_{\bm{x}} u (\bm{x};\bm{\theta}) \right|^2\Delta t}{N_s}.
\end{equation}
By taking the expectation over the sampling measure,  the bias of the empirical loss $\widetilde{\mathcal{L}}^{\Omega}(\bm{\theta})$ is 
\begin{equation}
\normalfont{\mbox{Bias}}_{\mathcal{L}^{\Omega}}\left[\widetilde{\mathcal{L}}^{\Omega} \right]{\quad} \simeq {\quad}\frac{\Delta t}{N_s} \mathbb{E}_{\bm{x}\sim \Omega}\left[| \nabla_{\bm{x}} u(\bm{x};\bm{\theta}) |^2 \right].
\end{equation}
\end{proof}

Theorem~\ref{thm:lossbias} indicates that neural network optimization with the empirical loss is implicitly regularized by the variance of the target samples, which can hinder convergence toward satisfying the martingale property. As the network approximation approaches the PDE solution, the target variance at each collocation point becomes largely determined by the local topology of the neural network: regions with larger gradient magnitudes induce higher target-sample variance. This variance is further controlled by two key parameters, which govern its overall scale.


When the time interval $\Delta t$ increases, each stochastic walker explores a broader neighborhood and gathers information over a larger region. While this expanded observation enriches the representation of the solution, it also amplifies the bias in the empirical loss. Mitigating this bias requires employing a sufficiently large number of walkers ($N_s$). Conversely, reducing $\Delta t$ lowers the bias so that fewer walkers are needed to achieve a comparable bias level. However, a smaller $\Delta t$ also restricts neighborhood exploration, thereby limiting the information available to the network during training.

From the variance estimation of the mean statistic in Eq.~\eqref{eq:each_target_sample_mean_variance}, we quantify the uncertainty of the empirical target value at each point $\bm{x}$ through the Chebyshev's inequality as 
\begin{equation}
\forall \epsilon >0, \hspace{3mm} \mathbb{P}\left(\left|\overline{y_{\bm{x}}}- \mathbb{E}\left[\overline{y_{\bm{x}}} \right] \right| > \epsilon \right) \lesssim \frac{|\nabla_{\bm{x}} u(\bm{x};\bm{\theta})|^2\Delta t}{\epsilon^2 N_s}. 
\end{equation}

\begin{cor}
The empirical loss $\widetilde{\mathcal{L}}^{\Omega}(\bm{\theta})$ in Eq.~\eqref{eq:q_loss_discrete} is an asymptotically unbiased estimator of the exact martingale loss $\mathcal{L}^{\Omega}(\bm{\theta})$ in Eq.~\eqref{eq:q_loss_continuous} with respect to both the {time interval} (i.e., $\Delta t \rightarrow 0$) and the number of stochastic walkers (i.e., $N_s \rightarrow \infty$).
\end{cor}

\subsection{Existence of a lower bound for the {time interval} of the Feynman-Kac formulation}
We now focus on the trainability issue for a small $\Delta t$. For a $k$-dimensional vector $\bm{\mu}\in \mathbb{R}^{k}$ and a positive value $\sigma^2 \in \mathbb{R}^{+}$, we detnote $f_{\bm{\mu}, \sigma^2}$ as the probability density function (PDF) of multivariate normal distribution with mean $\bm{\mu}$ and covariance matrix $\sigma^2 \mathbb{I}$ where $\mathbb{I}$ is the identity matrix in $\mathbb{R}^{k\times k}$. 
Hereafter, for notational simplicity, we suppress the dependence of the advection and force terms on $u(\bm{x})$, that is, $\bm{V}(\bm{x})=\bm{V}(\bm{x},u(\bm{x}))$ and $G(\bm{x})=G(\bm{x},u(\bm{x}))$. We first need the following lemma to analyze the trainability issue concerning $\Delta t$.

\begin{lem}
The numerical target evaluation using $\tilde{q}$-martingale in Eq.~\eqref{eq:q_tilde_martingale_process} for a small {time interval} $\Delta t$ is decomposed into a convolution with a normal distribution and the force effect 
\begin{equation}\label{eq:qtilde_target_eval}
\mathbb{E}[\tilde{q}(\Delta t; u, \bm{x},\bm{B}_{\Delta t})]=(u * f_{\bm{V}(\bm{x})\Delta t, \Delta t})(\bm{x}) -G(\bm{x})\Delta t.
\end{equation}
\end{lem}
\begin{proof}\renewcommand{\qedsymbol}{}
For a small {time interval} $\Delta t$, we consider the approximation of the stochastic integrals associated with the $\tilde{q}$-martingale as follows:
\begin{equation}
\mathbb{E}[\tilde{q}(\Delta t; u, \bm{x},\bm{B}_{\Delta t})]= \mathbb{E}\left[\left(u(\bm{B}_{\Delta t}) -G(\bm{x})\Delta t\right)\exp\left(\bm{V}(\bm{x}))\cdot \Delta\bm{B}_{\Delta t} -\frac{1}{2}|\bm{V}(\bm{x}))|^2\Delta t \right)\middle |\bm{B}_0=\bm{x} \right].
\end{equation}
Since the standard Brownian motion follows $\bm{B}_{\Delta t} \sim \mathcal{N}(\bm{x}, \Delta t\mathbb{I})$,
\begin{align}
\mathbb{E}[\tilde{q}(\Delta t; u, \bm{x},\bm{B}_{\Delta t})]&=\int_{\bm{y}\in\mathbb{R}^{d}}(u(\bm{y})-G(\bm{x})\Delta t)\exp\left(\bm{V}(\bm{x})\cdot (\bm{y}-\bm{x})-\frac{1}{2}|\bm{V}(\bm{x})|^2\Delta t \right)f_{\bm{x},\Delta t}(\bm{y})\bm{d}\bm{y} \\
&=\int_{\bm{z}\in\mathbb{R}^{d}}(u(\bm{x}-\bm{z})-G(\bm{x})\Delta t)\exp\left(\bm{V}(\bm{x})\cdot \bm{z}-\frac{1}{2}|\bm{V}(\bm{x})|^2\Delta t \right)f_{\bm{0},\Delta t}(\bm{z})\bm{d}\bm{z} \\
&= \int_{\bm{z}\in\mathbb{R}^{d}}(u(\bm{x}-\bm{z})-G(\bm{x})\Delta t)f_{\bm{V}(\bm{x})\Delta t, \Delta t}(\bm{z})\bm{d}\bm{z} \\
&= \int_{\bm{z}\in\mathbb{R}^{d}}u(\bm{x}-\bm{z})f_{\bm{V}(\bm{x})\Delta t, \Delta t}(\bm{z})\bm{d}\bm{z} -G(\bm{x})\Delta t.
\end{align}
where the third equality holds by the algebraic property 
\begin{equation}\bm{V}(\bm{x})\cdot \bm{z}-\frac{1}{2}|\bm{V}(\bm{x})|^2\Delta t -\frac{1}{2\Delta t} \bm{z}\cdot \bm{z}=-\frac{1}{2\Delta t}(\bm{z}-\bm{V}(\bm{x})\Delta t)\cdot (\bm{z}-\bm{V}(\bm{x})\Delta t)
\end{equation}
in the exponent.
\end{proof}
We note that for a nontrivial advection field $\bm{V}(\bm{x})$, the convolution is inhomogenous over the domain as, for each $\bm{x}$, it takes account for the neighborhood shifted toward the advection vector $\bm{V}(\bm{x})$ for a small {time interval} $\Delta t$ using the normal density function $\mathcal{N}(\bm{V}(\bm{x})\Delta t, \Delta t\mathbb{I})$. Instead of using the standard Brownian walkers $\bm{B}_t$, the stochastic process $\bm{X}_t$ in Eq.~\eqref{eq:q_stochastic_walkers} directly reflects the shift toward the field direction in the random sampling.  

\begin{remark}
The numerical target evalution using $q$-martingale in Eq.~\eqref{eq:q_stochastic_walkers} and Eq.~\eqref{eq:q_martingale_stochastic_process} has the same representation as Eq.~\eqref{eq:qtilde_target_eval}.
\end{remark}

The bootstrapping approach in DFLM demonstrated in Eq.~\eqref{eq:sgd_update} and Eq.~\eqref{eq:q_loss_discrete} is to update the neural network \textit{toward} the pre-evaluated target value using the current state of the neural network. That is, 
\begin{equation}
u(\bm{x};\bm{\theta}_{n+1}) \leftarrow (u(\cdot;\bm{\theta}_n) * f_{\bm{V}(\bm{x})\Delta t, \Delta t})(\bm{x}) -G(\bm{x})\Delta t.
\end{equation}
To achieve the pre-evaluated target, it may require multiple gradient descent steps in Eq.~\eqref{eq:sgd_update}. For instance, when updating $\bm{\theta}_n$ from $\bm{\theta}_{n-1}$, $M$ number of additional gradient descent steps could be considered as 
\begin{equation}
\bm{\theta}_{n-1}^{(m+1)}=\bm{\theta}_{n-1}^{(m)}-\alpha \nabla \widetilde{\mathcal{L}}_n\left(\bm{\theta}_{n-1}^{(m)}\right),~m=0,1,\cdots,M-1, {\quad} \bm{\theta}_{n-1}^{(0)}=\bm{\theta}_{n-1},~\bm{\theta}_{n-1}^{(M)} = \bm{\theta}_n.
\end{equation}
In the subsequent analysis, we assume that the update $u(\cdot;\bm{\theta}_{n+1})$ is equal to the target value function evaluated using $u(\cdot;\bm{\theta}_n)$. Formally, we define the target operator $T_n:\mathcal{L}^2(\Omega) \rightarrow \mathcal{L}^2(\Omega)$ at the $n$-th iteration as 
\begin{equation}
T_n :\mathcal{L}^2(\Omega) \rightarrow \mathcal{L}^2(\Omega), ~~~ (Tu)(\bm{x})=(u * f_{\bm{V}(\bm{x})\Delta t, \Delta t})(\bm{x}) -G(\bm{x})\Delta t
\end{equation}
under the regularity assumptions $\bm{V}, G \in \mathcal{L}^2(\Omega)$. When $\bm{V}$ and $G$ are independent from $u$, $T_{n+1}=T_n, \forall n \in \mathbb{N}_0$. The learning of DFLM is understood as the recursion of the operator $T_n$ with an initial function $u_0=u(\cdot;\bm{\theta}_0)$ as 
\begin{equation}
u(\cdot;\bm{\theta}_{n+1}) = u_{n+1} =T_nu_{n} = T_nu(\cdot;\bm{\theta}_n), \forall n \in \mathbb{N}_0.
\end{equation} 

\begin{example}
\normalfont
For the Laplace equation $\Delta u = 0 $ in $\Omega$, the training in DFLM is the recursion of the convolution of the normal density function $f_{\bm{0}, \Delta t}$ as 
\begin{equation}
u_{n+1} =  u_n * f_{\bm{0}, \Delta t}, \forall n \in \mathbb{N}_0.
\end{equation}
When the {time interval} $\Delta t$ is large, the convolution considers a broader neighborhood around each collocation point $\bm{x}$, given that the density function exhibits a long tail. Conversely, for a smaller {time interval} $\Delta t$, the convolution considers a more localized neighborhood. When $\Delta t \rightarrow 0$, $f_{\bm{0}, \Delta t}(\bm{x}) \rightarrow \delta(\bm{x})$ in the sense of distribution, in which $u_{n+1}(\bm{x}) = (u_n * \delta(\bm{x}))(\bm{x})=u_n(\bm{x})$,$\forall \bm{x} \in \Omega$, $\forall n \in \mathbb{N}_0$. In this case, the training process does not advance while staying at the initial function $u_0$. 
\end{example}

The above example shows that the training procedure depends on the choice of the {time interval} $\Delta t$. Opting for an excessively small {time interval} $\Delta t$ can result in the target function being too proximate to the current function, potentially leading to slow or hindered training progress. We aim to quantify how much training can be done at each iteration depending on the {time interval} $\Delta t$. For this goal, we need a lemma for the normal distribution.

\begin{lem}\label{lemma:expactation_abs_normal}
For a $k$-dimensional normal random variable $\bm{w}=(w_1,w_2,\cdots, w_k)$ with mean $\bm{\mu}$ and variance $\sigma^2\mathbb{I}, \sigma\in\mathbb{R}$, the expectation of the absolute value of $\bm{w}$ is bounded as 
\begin{equation}
\mathbb{E}[|\bm{w}|] \leq C_1\sigma \exp\left(-\frac{|\bm{\mu}|^2}{2\sigma^2} \right) + C_2|\bm{\mu}| 
\end{equation}
where the constant $C_1=k\sqrt{\frac{2}{\pi}}$ and $C_2 = k$ are independent of $\bm{\mu}$ and $\sigma$.
\end{lem}

\begin{proof}\renewcommand{\qedsymbol}{}
Let  $f_{\mu_i, \sigma^2}$ be the density of the univariate normal with mean $\mu_i$ and variance $\sigma^2$. Also, let $\Phi$ be the cumulative distribution function (CDF) of the standard normal distribution. Since $w_i$, $i=1,2,\cdots,k$, are pairwise independent, the density function of $\bm{w}$ is equal to $\prod \limits_{i=1}^{k} f_{\mu_i, \sigma^2}(w_i)$. Thus, we have
\begin{align}
\mathbb{E}[|\bm{w}|]&=\int_{\mathbb{R}^d}|\bm{w}|\prod \limits_{i=1}^{k} f_{\mu_i, \sigma^2}(w_i)d\bm{w} \\
&\leq \int_{\mathbb{R}^d} \sum \limits_{j=1}^{d}|w_j|\prod \limits_{i=1}^{k} f_{\mu_i, \sigma^2}(w_i)d\bm{w} \\
&= \sum \limits_{j=1}^{d}\int_{\mathbb{R}^d} |w_j|\prod \limits_{i=1}^{k} f_{\mu_i, \sigma^2}(w_i)d\bm{w} \\
&= \sum \limits_{j=1}^{k}\int_{\mathbb{R}^d} |w_j| f_{\mu_j, \sigma^2}(w_j) \prod \limits_{i=1, i\neq j}^{k} f_{\mu_i, \sigma^2}(w_i)d\bm{w} \\
&= \sum \limits_{j=1}^{k}\int_{\mathbb{R}} |w_j| f_{\mu_j, \sigma^2}(w_j)dw_j \\
&= \sum \limits_{j=1}^{k} \sqrt{\frac{2}{\pi}} \sigma\exp\left(-\frac{\mu_j^2}{2\sigma^2}\right)+\mu_j\left[1-2\Phi\left(-\frac{\mu_j}{\sigma}\right) \right] \\
&\leq k\sqrt{\frac{2}{\pi}} \sigma\exp\left(-\frac{|\bm{\mu}|^2}{2\sigma^2}\right)+k|\bm{\mu}|,
\end{align}
where the second equality from the last holds as the sum of the expectations of the folded normal distributions. 
\end{proof}

Now, we are ready to prove the following theorem.
\begin{thm}\label{thm:pointwise_learning}
Let $\{u_n\}_{n=0}^{\infty}=\{u(\cdot;\bm{\theta}_n)\}_{n=0}^{\infty}$ be the sequence of the states of a neural network in the training procedure of DFLM starting from $u_0=u(;\bm{\theta}_0)$. We assume that $u_n \in C^{1}\left(\overline{\Omega}\right)$, $\forall n \in \mathbb{N}_0$. Then, the learning amount at each iteration measured by the pointwise difference in the consecutive states is approximated as 
\begin{equation}
|u_{n+1}(\bm{x})-u_n(\bm{x})| \leq \left|\nabla_{\bm{x}} u_n(\bm{x})\right| \left(C_1\sqrt{\Delta t} + C_2|\bm{V}(\bm{x})|\Delta t  \right) + |G(\bm{x})|\Delta t,
\end{equation} 
with constants $C_1=k\sqrt{\frac{2}{\pi}}$ and $C_2 = k$.
\end{thm}
\begin{proof}\renewcommand{\qedsymbol}{}
\begin{align}
|u_{n+1}(\bm{x})-u_n(\bm{x})| &= |(T_nu_n)(\bm{x})-u_n(\bm{x})|\\
 &= \int_{\bm{z}\in\mathbb{R}^{d}}|u_n(\bm{x}-\bm{z})-u_n(\bm{x})|f_{-\bm{V}(\bm{x})\Delta t, \Delta t}(\bm{z})\bm{d}\bm{z} +|G(\bm{x})|\Delta t \\
&= \int_{\bm{z}\in\mathbb{R}^{d}}\left|\nabla_{\bm{x}} u_n(\bm{x})\cdot \bm{x}^{\prime}(\bm{z})\right|f_{-\bm{V}(\bm{x})\Delta t, \Delta t}(\bm{z})\bm{d}\bm{z} +|G(\bm{x})|\Delta t, ~~~ \left|\bm{x}^{\prime}(\bm{z}) \right| \leq |\bm{z}| 
\end{align}
\begin{align}
\hspace{31mm}&\leq \int_{\bm{z}\in\mathbb{R}^{d}}\left|\nabla_{\bm{x}} u_n(\bm{x})\right| |\bm{z}| f_{-\bm{V}(\bm{x})\Delta t, \Delta t}(\bm{z})\bm{d}\bm{z} +|G(\bm{x})|\Delta t  \\
& =  \left|\nabla_{\bm{x}} u_n(\bm{x})\right| \left(\int_{\bm{z}\in\mathbb{R}^{d}} |\bm{z}| f_{-\bm{V}(\bm{x})\Delta t, \Delta t}(\bm{z})\bm{d}\bm{z}\right) +|G(\bm{x})|\Delta t \\
& \leq \left|\nabla_{\bm{x}} u_n(\bm{x})\right| \left(C_1\sqrt{\Delta t}\exp\left(-\frac{|\bm{V}(\bm{x})|^2}{2} \Delta t \right) + C_2|\bm{V}(\bm{x})|\Delta t  \right) + |G(\bm{x})|\Delta t\\
& \leq \left|\nabla_{\bm{x}} u_n(\bm{x})\right| \left(C_1\sqrt{\Delta t} + C_2|\bm{V}(\bm{x})|\Delta t  \right) + |G(\bm{x})|\Delta t,
\end{align}
where the third equality holds by the taylor expansion of $u_n$ at $\bm{x}$ and the second inequality from the last holds by Lemma~\ref{lemma:expactation_abs_normal}.
\end{proof}
Theorem \ref{thm:pointwise_learning} states that for each point $\bm{x}$, the learning from the convolution is proportional to i) the magnitude of the gradient, ii) $\mathcal{O}(\sqrt{\Delta t})$ from the shape of the normal distribution and iii) $\mathcal{O}(\Delta t)$ in the advection. Also, the learning from the forcing term is of order $\mathcal{O}(\Delta t)$.

\begin{cor}\label{cor:uniform_learning}
Let $\{u_n\}_{n=0}^{\infty}=\{u(\cdot;\bm{\theta}_n)\}_{n=0}^{\infty}$ be the sequence of the states of a neural network in the training procedure of DFLM starting from $u_0=u(\cdot;\bm{\theta}_0)$. We assume that $u_n \in C^{1}\left(\overline{\Omega}\right)$, $\forall n \in \mathbb{N}_0$ and $\bm{V}, G \in C(\overline{\Omega})$. Then 
\begin{equation}
\|u_{n+1}-u_n\|_2 \leq (C_1\sqrt{\Delta t}+ C_2\|\bm{V}\|_{\infty}\Delta t) \|\nabla_{\bm{x}} u_n\|_2 + \Delta t\|G\|_2.
\end{equation}
In particular, if $\|\nabla u_n\|_2$ is uniformly bounded, $\|u_{n+1}-u_n\|_2 \rightarrow 0$ in order $\mathcal{O}(\sqrt{\Delta t})$ as $\Delta t \rightarrow 0$.
\end{cor}

The analysis underscores the critical role of choosing an appropriate time interval $\Delta t$ to ensure that stochastic walkers adequately explore local neighborhoods, thereby enabling effective training. As shown in Theorem \ref{thm:pointwise_learning} and Corollary \ref{cor:uniform_learning}, the required magnitude of $\Delta t$ depends not only on the advection field $\bm{V}$ and the forcing term $G$, but also on the neural network’s topology, which is tied to the underlying PDE solution. At the same time, increasing $\Delta t$ introduces additional bias into the loss function (Theorem \ref{thm:lossbias}), which can hinder convergence. A practical way to counteract this effect is to increase the number of stochastic walkers $N_s$, thereby reducing the variance and mitigating the impact of the bias during training.

\section{Numerical Experiments}\label{sec:experiment}

In this section, we present numerical examples to validate the analysis of DFLM with respect to the {time interval} $\Delta t$ and the number of stochastic walkers $N_s$. Other parameters (e.g., $N_r$ and $N_b$) are chosen sufficiently large to minimize their influence, allowing us to isolate the effects of $\Delta t$ and $N_s$.
As test problems, we consider (i) the Poisson equation and (ii) the Taylor–Green vortex fluid problem. For the Poisson equation, in addition to being a standard benchmark for numerical methods, the $q$- and $\tilde{q}$-martingales coincide due to the absence of an advection term. This identification allows direct sampling from the normal distribution without solving the stochastic differential equation, providing significant computational savings. Such efficiency enables us to perform extensive tests across a wide range of parameter variations.
The Taylor–Green vortex problem, in contrast, is a nonlinear PDE governed by the Navier–Stokes equations and involves explicit time dependence. Here, the presence of advection prevents martingale identification, and the stochastic differential equations for the walkers must be solved explicitly. As a result, the computational cost is considerably higher than in the Poisson case. Although we do not perform as extensive a parameter sweep for this problem, DFLM exhibits qualitatively similar behavior in both test cases.

\subsection{Poisson problem}\label{subsec:experimentsetup}
We solve the Poisson equation in the unit square $\Omega=(-0.5,0.5)^2$ with the homogeneous Dirichlet boundary condition\footnote{The homogeneous Dirichlet boundary condition minimizes the effect of the boundary treatment, which is not the interest of the current study.}
\begin{align}\label{eq:Poisson}
\begin{split}
\Delta u &= f \quad\textrm{in}\quad \Omega,  \\
u &= 0 \quad\textrm{on}\quad \partial \Omega.
\end{split}
\end{align}
We choose $f$ to be $-(2m\pi)^2\sin(2m\pi x_1)\sin(2m\pi x_2)$ so that the exact solution is 
\begin{equation}\label{eq:expsoln}
u(\bm{x})=\sin(2m\pi x_1)\sin(2m\pi x_2),\quad m \in \mathbb{N}.
\end{equation}
The empirical loss function at the $n$-th iteration for updating the neural network $\bm{\theta}_n$ is
\begin{equation}
\widetilde{\mathcal{L}}^{\Omega}_{n}(\bm{\theta}) = \frac{1}{N_r} \sum \limits_{i=1}^{N_r}\left| u(\bm{x}_i;\bm{\theta})-\frac{1}{N_s}\sum\limits_{j=1}^{N_s}\left\{ u\left(\bm{B}_{\Delta t}^{(i,j)};\bm{\theta}_{n-1} \right)-\int_0^{\Delta t}\frac{1}{2}f\left(\bm{B}_s^{(i,j)} \right) ds \right\} \right |^2.
\end{equation}
Here we use $N_r$ random sample collocation points $\{\bm{x}_i: 1\leq i \leq N_r\}$ and $N_s$ Brownian walkers $\left\{\bm{B}^{(i,j)}_s:\bm{B}^{(i,j)}_{0}=\bm{x}_i, 1\leq i \leq N_r, 1\leq j \leq N_s\right\}$ for each $\bm{x}_i$. To minimize the error in calculating the term related to $f$ and handling the boundary treatment, we use a small time step $\delta t\leq  \Delta t$ to evolve the discrete Brownian motion by the Euler-Maruyama method
\begin{equation}
\bm{B}_{m \delta t} = \bm{B}_{(m-1) \delta t} + \sqrt{\delta t} \bm{Z}, {\quad} \bm{Z} \sim \mathcal{N}(0, \mathbb{I}_2), ~ m \in \mathbb{N}.
\end{equation}
We note again that we can quickly draw samples from $\sqrt{\delta t}\bm{Z}$ and add to $\bm{B}_{(m-1)\delta t}$, which can be computed efficiently. Using these Brownian paths, the stochastic integral during the time period $[0, \Delta t]$, $\Delta t = M \delta t$, $M \in \mathbb{N}$, is estimated as
\begin{equation}
\int_0^{\Delta t}\frac{1}{2}f\left(\bm{B}_s^{(i,j)} \right) ds \simeq \sum \limits_{m=0}^{M-1}\frac{1}{2}f\left(\bm{B}^{(i,j)}_{m\delta t}\right)\delta t.
\end{equation}
\begin{figure}[!ht]
\centering
\includegraphics[width=1.0\textwidth]{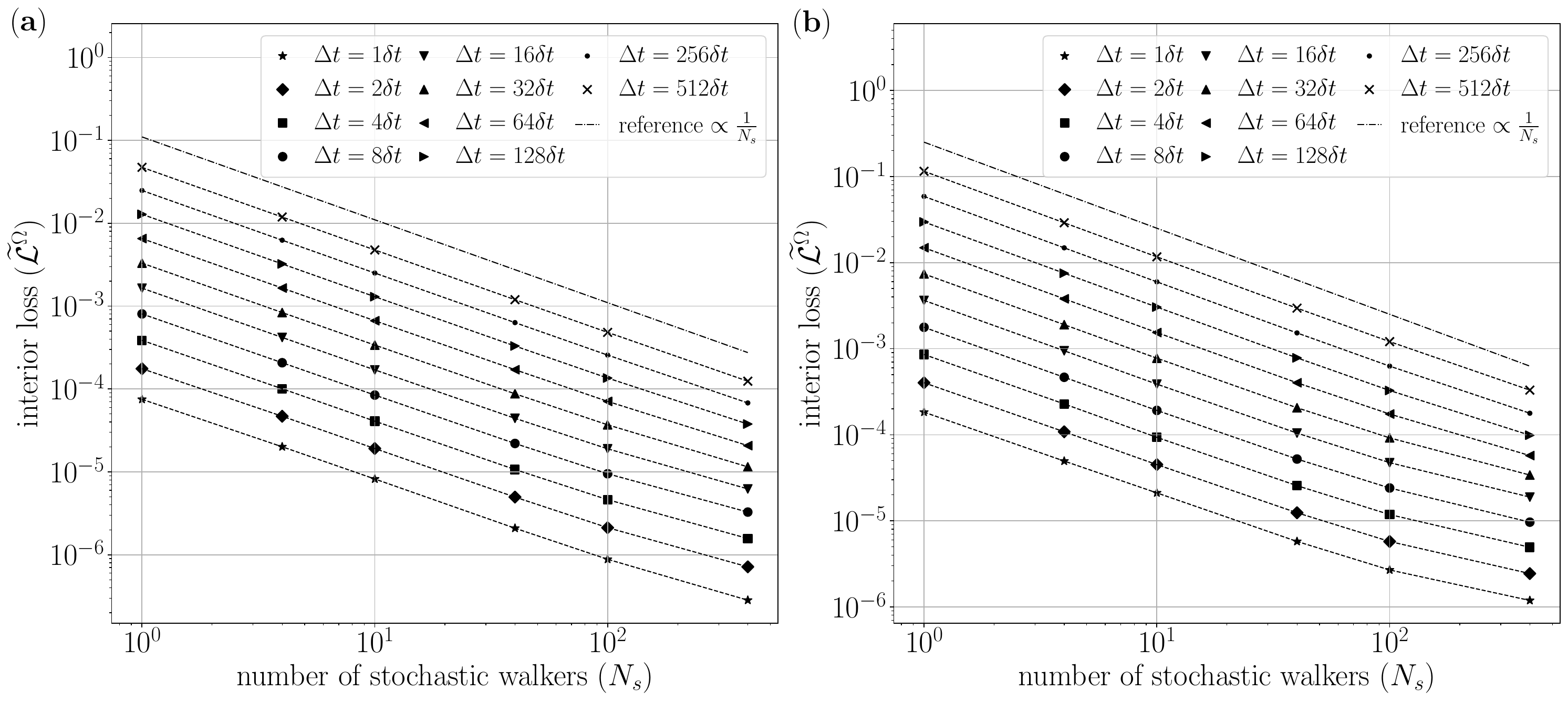}
\caption{Empirical interior training loss for various walker size $N_s$ (horizontal axis) and {time interval} $\Delta t$ (different line types). (a) $m=1$ and (b) $m=3$.}\label{fig:lossNs}
\end{figure}
We impose the homogeneous Dirichlet boundary condition on the stochastic process $\bm{B}_t$ by allowing it to be absorbed to the boundary $ \partial \Omega$ at the exit position. We estimate the exit position and the time by linear approximation within a short time period. When the simulation of a Brownian motion comes across the boundary between the time $m\delta t$ and $(m+1)\delta t$, we estimate the exit information $t_{\text{exit}}$, $t_{\text{exit}}\in [m\delta t, (m+1)\delta t]$ and $\bm{B}_{t_{\text{exit}}}$  by the intersection of line segment between $\bm{B}_{m \delta t}$ and $\bm{B}_{(m+1)\delta t}$ and the boundary $\partial \Omega$. Once a walker is absorbed, the value of the neural network at the exit position $u(\bm{B}_{t_{\text{exit}}};\bm{\theta})$ in the target computation is set to $0$, the homogeneous boundary value, and the time step $\delta t$ in the integral approximation is replaced by $(t_{\text{exit}}-m\delta t)$. To enhance the information on the boundary, we also impose the boundary loss term
\begin{equation}
\widetilde{\mathcal{L}}^{\partial \Omega}(\bm{\theta})  = \frac{1}{N_b}\sum \limits_{l=1}^{N_b} \left| u(\bm{x}_l;\bm{\theta}) - g(\bm{x}_l)\right|^2
\end{equation}
using $N_b$ boundary random collocation points.

To investigate the dependency of training trajectory on the choice of the {time interval}  $\Delta t$ and the number of stochastic walkers $N_s$, we solve the problem with various combinations of these two parameters while keeping the other parameters fixed.
Regarding the network structure, we use a standard multilayer perceptron (MLP) with three hidden layers, each comprising $200$ neurons and employing the ReLU activation function. The neural network is trained using the Adam optimizer \cite{adam} with learning parameters $\beta_1 = 0.99$ and $\beta_2 = 0.99$. At each iteration, we randomly sample  $N_r=2000$ interior and $N_b=400$ boundary points from the uniform distribution. We consider ten different {time interval}s $\Delta t= 2^{p}$, $p=0,1,2,\cdots, 9$ and six different stochastic walker sizes $N_s=1,4,10,40,100,400$, resulting in a total of sixty combinations of $(\Delta t, N_s)$. In considering the randomness of the training procedure, we run ten independent trials with extensive $1.5 \times 10^5$ iterations for each parameter pair, which guarantees the convergence of the training loss.

\begin{figure}[!ht]
\centering
\includegraphics[width=1.0\textwidth]{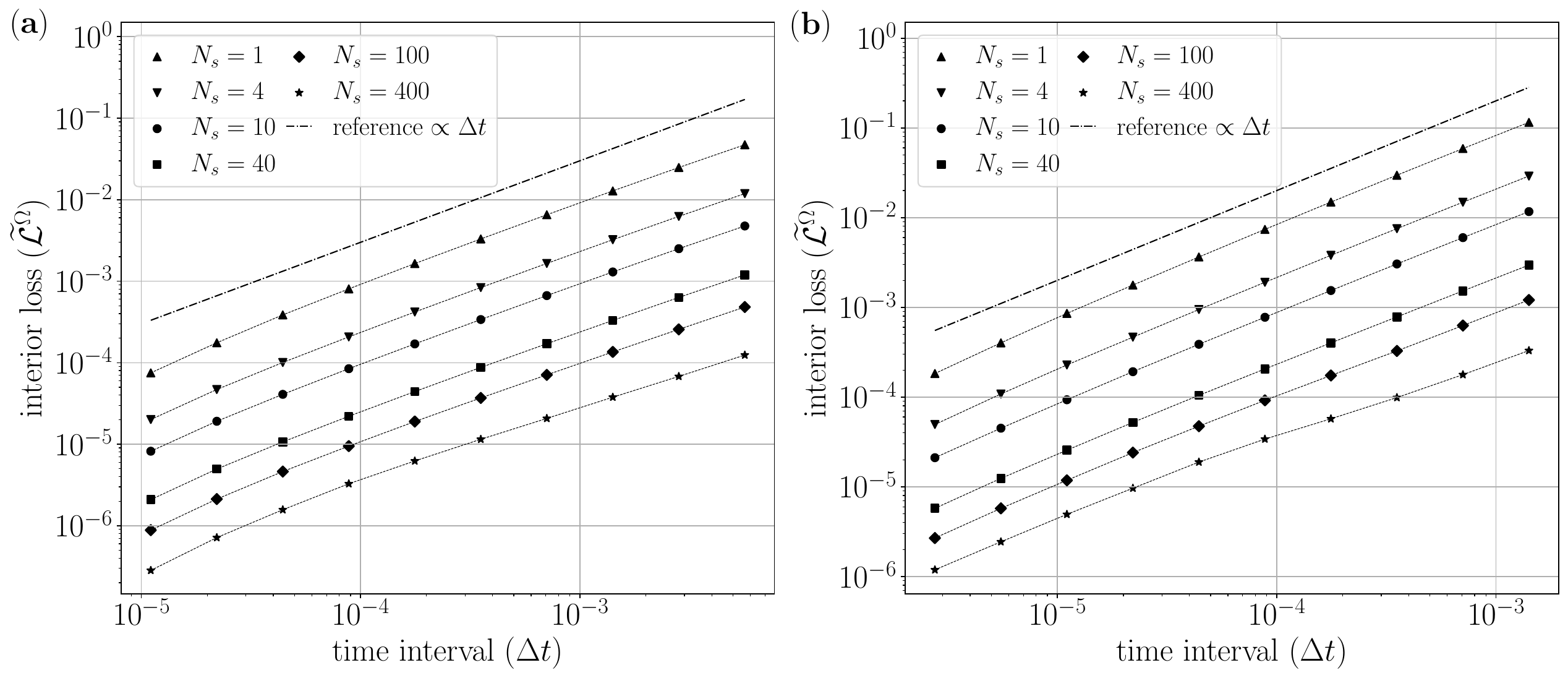}
\caption{Empirical interior training loss for various {time interval} $\Delta t$ (horizontal axis) and walker size $N_s$ (different line types). (a) $m=1$ and (b) $m=3$.}\label{fig:lossDt}
\end{figure}
We use the average interior training loss out of 10 trials to measure the training loss bias in Theorem \ref{thm:lossbias},  which we call `training loss.' We also consider two problems with different wavenumber $m=1$ and $3$ for Eq.~\eqref{eq:expsoln}. We consider two solutions with $m=1$ and $m=2$ to check the contribution from the different magnitude $\ell_2$ norms of the gradient. 

Figure \ref{fig:lossNs} shows the log-log plot of the training loss after convergence as a function of the walker size $N_s$ (horizontal axis) with various $\Delta t$ (different line types). Figure \ref{fig:lossNs} (a) and (b) are the cases with $m=1$ and $3$, respectively, and we can see that the training loss has a larger value for the more complicated case $m=3$  (about three times larger than the case of $m=1$), which the gradient magnitudes can explain for $m=1$ and $m=3$. As the analysis in the previous section predicts, the training loss decreases as the walker size increases for both cases, which aligns with the reference line of $\frac{1}{N_s}$ (dash-dot). In comparison between different line types, we can also check that the training loss decreases as the {time interval} $\Delta t$ decreases. The (linear) dependence of the training loss on $\Delta t$ is more explicit in Figure \ref{fig:lossDt}. Figure \ref{fig:lossDt} shows the training loss as a function of $\Delta t$ (horizontal axis) with various $N_s$ (different line types). As in the previous figure, Figure \ref{fig:lossDt} (a) and (b) show the results for the solution with $m=1$ and $3$, respectively. Compared to the reference line of $\Delta t$ (dash-dot), all training losses show a linear increase as $\Delta t$ increases.

We now check the test error after the training loss converges. In particular, we use the relative $\mathcal{L}^2$ error as a performance measure, calculated using the $1001\times 1001$ uniform grid. As discussed in the previous section, a small training loss does not always imply a small test error. In DFLM, if $\Delta t$ is sufficiently small, the left- and right-hand sides of Eq.~\eqref{eq:q-martingale} get close enough that the training loss can be small for an arbitrary initial guess, which fails to learn the PDE solution. 

\begin{figure}[!ht]
\centering
\includegraphics[width=1.0\textwidth]{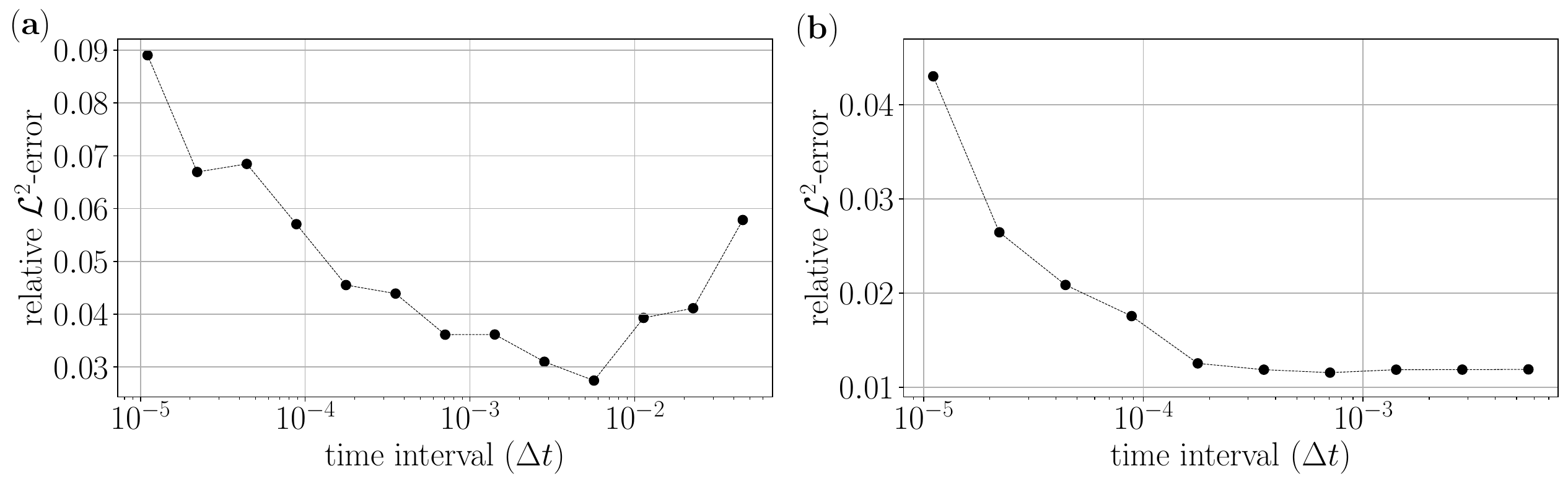}
\caption{Relative $\mathcal{L}^2$ test error for varying {time interval} $\Delta t$. (a) $N_s=1$ (b) $N_s=400$.}\label{fig:errDt}
\end{figure}
\begin{figure}[!ht]
\centering
\includegraphics[width=1.0\textwidth]{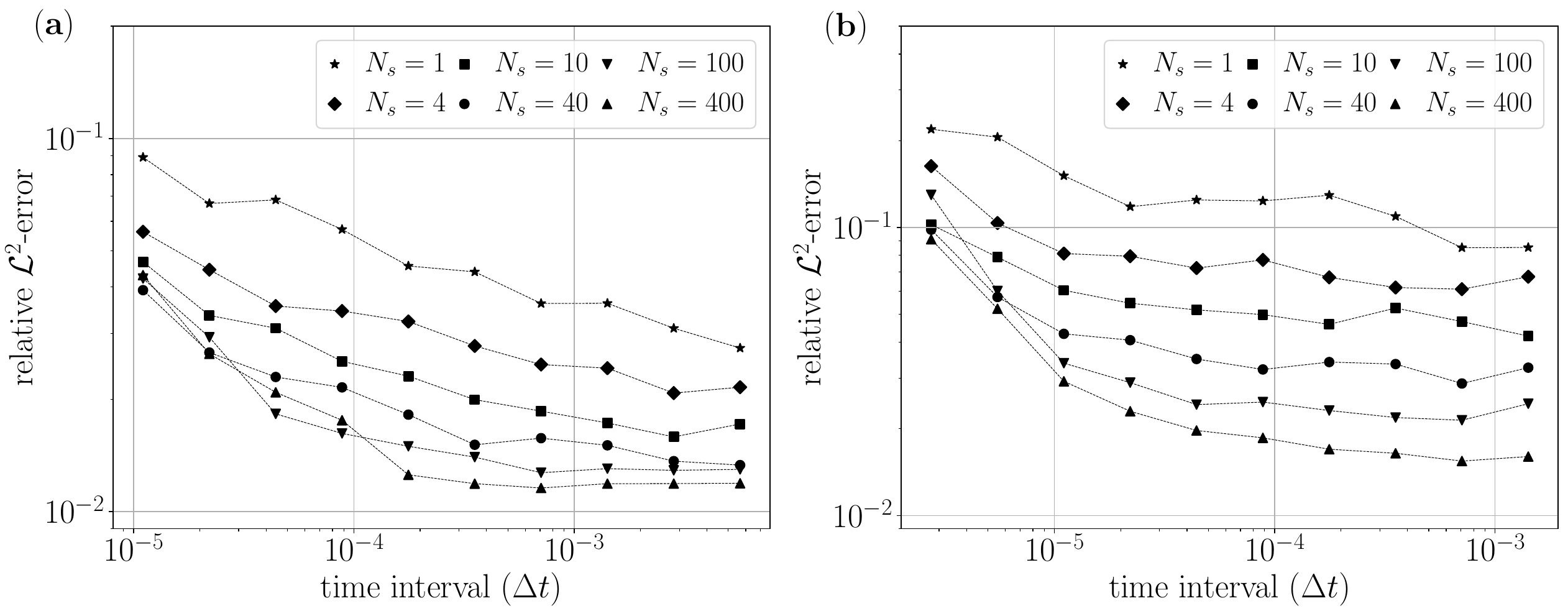}
\caption{Relative $\mathcal{L}^2$ test error as a function of {time interval} $\Delta t$ for various $N_s$ values. (a) $m=1$ and (b) $m=3$.}\label{fig:errDtNs}
\end{figure}
For the case of $N_s=1$ and $400$, Figure \ref{fig:errDt} shows the relative $\mathcal{L}^2$ test error as $\Delta t$ increases for the solution Eq.~\eqref{eq:expsoln} with $m=1$. When $\Delta t$ is small ($\leq 3\times 10^{-3}$), on the other hand, we can check that the training loss bias cannot explain the performance anymore. The test error increases as $\Delta t$ decreases regardless of the size of $N_s$. In other words, the result shows that the {time interval} $\Delta t$ must be sufficiently large for the network to learn the PDE solution by minimizing the training loss. When the training loss bias makes a non-negligible contribution with $N_s=1$, the test error can obtain the minimal value with an optimal $\Delta t\approx 5\times 10^{-2}$. If the bias contribution is small with $N_s=400$, the test error will be minimal with $\Delta t\geq 3\times 10^{-3}$. 
The increasing test error for decreasing $\Delta t$ is similar for other values of $N_s$ and solution types. Figure \ref{fig:errDtNs} shows the test relative error as a function of $\Delta t$ for the two solutions with $m=1$ (a) and $3$ (b) for various values of $N_s$. As $\Delta t$ decreases, the test error increases. Also, as $N_s$ decreases, which increases the sampling error in calculating the expectation, the test error increases for both solutions.

From Figure \ref{fig:errDtscale1n3}, which shows the test error of both solutions with $N_s=400$ and various $\Delta t$ values, we can also see that the optimal $\Delta t$ is related to the local variations of the solution. First, as we mentioned before, the solution with $m=3$ has a larger error as its gradient $\ell_2$ norm is larger than that of $m=1$, related to the loss bound and training update. Also, we use the same network structure, and all other parameters are equal for both solutions. Thus, it is natural to expect a much larger test error in the more complicated solution with $m=3$ than in the case of $m=1$. In comparison between the two solutions, we find that the test error of the more oscillatory solution ($m=3$) stabilizes much faster for a small $\Delta t$. That is, even using a small $\Delta t$, which yields a small neighborhood to explore and average, the more oscillatory solution case can see more variations than the simple solution case. Thus, the training loss can lead to a trainable result. 
\begin{figure}[!ht]
\centering
\includegraphics[width=.8\textwidth]{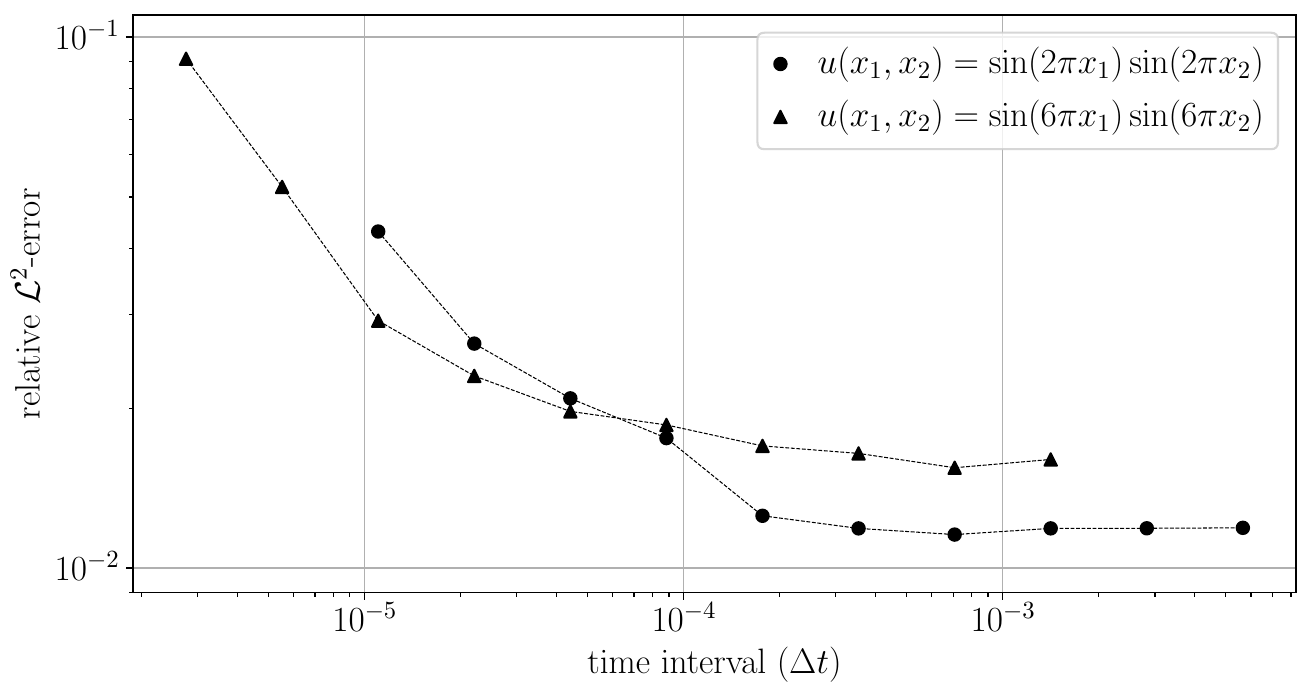}
\caption{Relative $\mathcal{L}^2$ test error as a function of {time interval} $\Delta t$ for the two solutions with $m=1$ (simple) and $m=3$ (more oscillatory). $N_s$ is fixed at $400$. }\label{fig:errDtscale1n3}
\end{figure}
Quantitatively, the optimal {time interval}s between the two solutions differ by a factor of about ten (that is, the more oscillatory solution can use $\Delta t$ ten times smaller than the one of the simple solution). This difference can be explained by the fact that the variance of the stochastic walkers is proportional to $\Delta t$, or the standard deviation is proportional to $\sqrt{\Delta t}$. As the more oscillatory solution has a wavenumber three times larger than the simple case, we can see that nine times shorter $\Delta t$ will cover the same variations as in the simple case, which matches the numerical result.

\subsection{Taylor-Green vortex problem}
To further validate our analysis, we consider the Taylor–Green vortex problem within the DFLM framework introduced in \cite{DFLMflow}. Specifically, we solve an initial value problem of incompressible Navier–Stokes equations in the unit square $\Omega = (0,1)^2$
\begin{eqnarray}
\label{eq:NSE}\frac{\partial \bm{u}}{\partial t} + \bm{u}\cdot \nabla\bm{u} &=& -\frac{1}{\rho} \nabla p + \nu \Delta \bm{u} + \bm{f}, ~~\text{in}~\Omega \times (0,T],\\
\label{eq:Incompressibility}\nabla \cdot \bm{u} &=& 0 ~~\text{in}~\Omega \times (0,T],\\
\bm{u}(\bm{x},0)&=&\bm{g}(\bm{x}),
\end{eqnarray}
where $\bm{u}(\bm{x},t) \in \mathbb{R}^2$ denotes the incompressible velocity field, $p(\bm{x},t)\in\mathbb{R}$ the pressure, $\rho$ the fluid density, $\nu$ the kinematic viscosity, and $\bm{f}$ an external forcing term. The boundary condition is periodic in both directions,
\begin{equation}
\bm{u}(x_1+1, x_2, t) = \bm{u}(x_1, x_2, t), \hspace{5mm} \bm{u}(x_1, x_2+1, t) = \bm{u}(x_1, x_2, t).
\end{equation}
We consideer the Taylor-Green problem with the following initial value 
\begin{equation}
\bm{u}(\bm{x}, 0)= (\sin(2\pi x_1) \cos(2\pi x_2), -\cos(2\pi x_1) \sin(2\pi x_2)) 
\end{equation}
and vanishing forcing term $\bm{f} = 0$. Under this setting, the exact solution is known as \cite{taylor1937mechanism} 
\begin{eqnarray}
\nonumber u_1(x_1,x_2,t) &=& \sin(2\pi x_1)\cos(2\pi x_2)e^{-8\pi^2 \nu t}, \\
\label{eq:TGanalyticsoln}u_2(x_1,x_2,t) &=& -\cos(2\pi x_1)\sin(2\pi x_2)e^{-8\pi^2 \nu t}, \\
\nonumber p(x_1,x_2,t) &=& -\frac{1}{4}(\cos(4\pi x_1)+\cos(4\pi x_2))e^{-16\pi^2 \nu t}.
\end{eqnarray}

The DFLM formulation for incompressible flows to handle incompressibility and time dependence is as follows. Following the idea in \cite{DFLMflow}, the solution of Eq.~\eqref{eq:NSE}-\eqref{eq:Incompressibility} admits the stochastic representation
\begin{equation}\label{eq:fluid_stochastic_representation}
\bm{u}(\bm{x},t) = \mathbb{E} \left[ \bm{u}(\bm{X}_{\Delta t}, t-{\Delta t})+ \int_0^{{\Delta t}} \bm{f}(\bm{X}_r, t-r)dr \bigg |\bm{X}_0=\bm{x} \right],
\end{equation}
where 
\begin{equation}
d\bm{X}_r = -\bm{u}(\bm{X}_r, t-r)dr + \sqrt{2\nu}d\bm{B}_r.
\end{equation}
The divergence free constraint Eq.~\eqref{eq:Incompressibility} is enforced by introducing a vector potential network $\bm{A}(\bm{x},t;\bm{\theta})$, where the velocity field is represented as 
\begin{equation}
\bm{u}(\cdot, \bm{\theta})=\nabla_{\bm{x}} \times \bm{A}(\cdot, \bm{\theta}).
\end{equation}
In contrast to the elliptic case discussed earlier, the stochastic representation in Eq.~\eqref{eq:fluid_stochastic_representation} links the physical time, $t$, of the governing equation to the stochastic time, $r$, associated with the walkers’ trajectories. The temporal evolution of the solution is determined by the historical flow states and external forcing, since the velocity at a current point $(\bm{x}, t)$ depends on the past values of $\bm{u}$ and $\bm{f}$ over a stochastic duration $s$, evolving backward in physical time. Specifically, the term $\bm{u}(\bm{X}_{\Delta t}, t-{\Delta t})$ accounts for the contribution of the past velocity field at time $t-{\Delta t}$, sampled in the spatial neighborhood reached by the backward walker $\bm{X}_{\Delta t}$, while the integral term $\int_0^{\Delta t} \bm{f}(\bm{X}_r, t-r),dr$ accumulates the effects of the forcing along the stochastic path $\bm{X}_r$ over the physical time interval $[t-{\Delta t}, t]$. We note that the physical time increment $\Delta t$ used to advance the solution coincides with the duration of the walkers’ stochastic evolution.
Using this solution representation, the training procedure follows the standard DFLM framework, where the loss function measures the discrepancy in the stochastic representation
\begin{equation}
\mathcal{L}(\bm{\theta})=\mathbb{E}\left[\left|\bm{u}(\bm{x,t};\bm{\theta}) - \mathbb{E} \left[ \bm{u}(\bm{X}_{\Delta t}, t-{\Delta t};\bm{\theta})+ \int_0^{\Delta t} \bm{f}(\bm{X}_r, t-r)dr \bigg |\bm{X}_0=\bm{x} \right]\right|^2\right]
\end{equation}
For the periodic boundary condition, walkers that cross the boundary are treated as re-entering continuously from the correpsonding opposite side (see, section 4.1 in \cite{DFLM}).

\begin{figure}[!ht]
\centering
\includegraphics[width=1.0\textwidth]{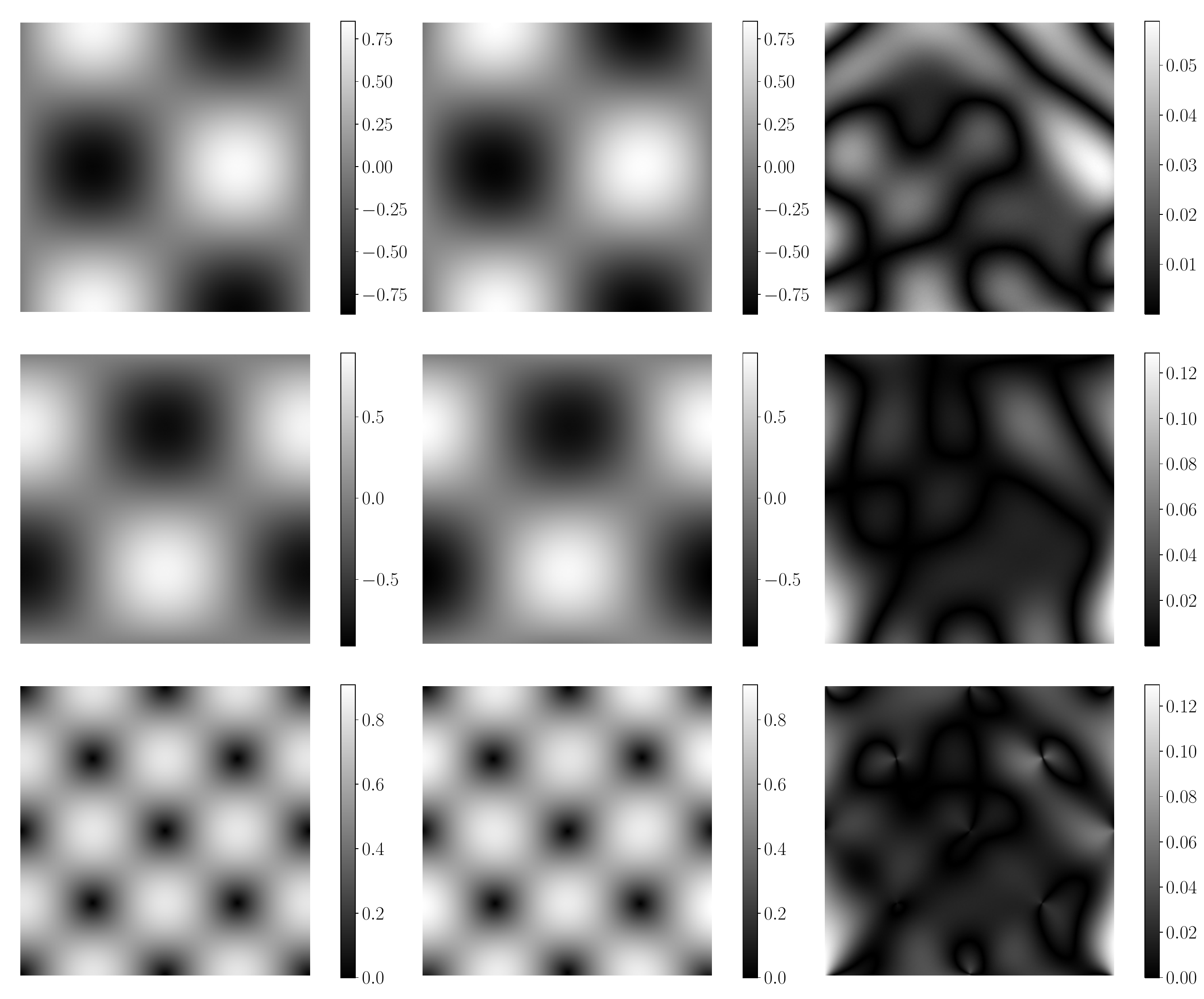}
\caption{Taylor–Green vortex approximation. Rows show $u_1$, $u_2$, and velocity magnitude $|\bm{u}|=\sqrt{u_1^2+u_2^2}$; columns show the exact solution, DFLM approximation, and pointwise error, respectively.}\label{fig:fluid_approx}
\end{figure}
The neural network used in this experiment is a standard multilayer perceptron (MLP) with three hidden layers of 200 neurons each and $\tanh$ activation functions. We sample $N_r=4000$ interior collocation points and employ $N_s=600$ stochastic walkers per point. Under this setup, the problem is solved up to $T=0.5$ using multiple time steps of the form $\Delta t = 3^{p}\delta t$, $p=0,1,2,3,4,5$, with $\delta t = 10^{-5}$. For reference, Fig.~\ref{fig:fluid_approx} compares the analytic solution (first column) with the DFLM solution obtained using $\Delta t = 3^{3}\delta t$ (second column). In comparison with the analytic solution Eq.~\eqref{eq:TGanalyticsoln}, DFLM captures the patterns of each components of the velocity fields, $u_1$ (first row) and $u_2$ (second row), along with velocity magnitude $|\bm{u}|$ (third row). The relative $\mathcal{L}^2$ errors obtained by different time steps are presented in Fig.~\ref{fig:fluid_errors}. In this test, DFLM has stable relative errors for $\Delta t\geq 3^{3}\delta t$. This result indicates again that sufficiently large time steps are required for DFLM to converge and stabilize for the fluid problem. 
\begin{figure}[!ht]
\centering
\includegraphics[width=1.0\textwidth]{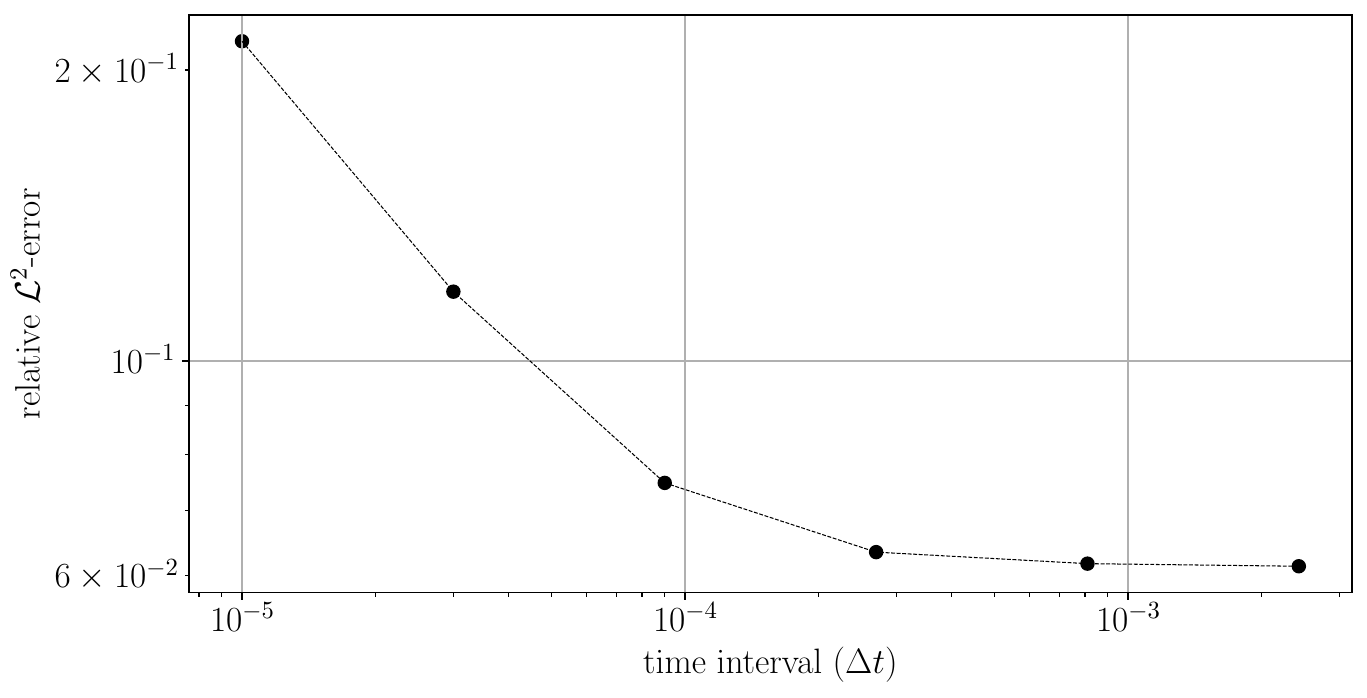}
\caption{Relative $\mathcal{L}^2$ test errors as a function of time interval $\Delta t$ in the Talyor-Green Vortex problem.}\label{fig:fluid_errors}
\end{figure}

\section{Discussions and conclusions}\label{sec:discussion}
The derivative-free loss method (DFLM) leverages the Feynman–Kac formulation to train neural networks for solving PDEs of the form Eq.~\eqref{eq:Quasi-linear elliptic}, including the Navier–Stokes equations Eq.~\eqref{eq:NSE}. In this study, we analyzed the bias of the empirical training loss and established that the loss becomes asymptotically unbiased as the number of walkers $N_s$ at each collocation point increases. The analysis further revealed that the bias grows proportionally with the time interval $\Delta t$, which governs how long stochastic walkers evolve to compute expectations. At the same time, we showed that $\Delta t$ must be sufficiently large to produce meaningful updates to the training loss; otherwise, the walkers fail to capture local variations in the solution. The numerical experiments for the Poisson and Tayler-Green vortex problems confirmed the existence of a problem-dependent lower bound on $\Delta t$, which reflects the local variability of the PDE solution. From a computational perspective, the results indicate that an efficient strategy is to identify the optimal lower bound of $\Delta t$ and then select $N_s$ as small as possible relative to this choice.

While our findings highlight the interplay between $\Delta t$ and $N_s$, an explicit quantitative method for determining the optimal lower bound of $\Delta t$ remains open. We expect this lower bound to depend strongly on solution characteristics, such as local oscillations or multiscale features. Extending the current analysis to multiscale PDEs is a natural next step. In such settings, the lower bound of $\Delta t$ is expected to shrink as oscillatory behavior intensifies. To address this, we envision adaptive or hierarchical time-stepping strategies inspired by hierarchical learning frameworks \cite{HiPINN}. By assigning different time intervals to distinct scale components of the solution and incorporating a hierarchical training procedure, one may accelerate convergence while preserving accuracy. Developing and testing such strategies is an important direction for future research.

\section*{Acknowledgments}
This work was supported by ONR MURI N00014-20-1-2595. 

\bibliographystyle{siam}
\bibliography{dflm_analysis}

\end{document}